\documentclass{article}
\usepackage[T1]{fontenc}
\usepackage[utf8]{inputenc}
\usepackage{amsmath, amsthm, amssymb, mathtools}
\usepackage{todonotes}
\usepackage{thm-restate}
\usepackage{xcolor}
\usepackage{latexcolors}
\usepackage{cite}
\usepackage[colorlinks=true, linkcolor=blue!70!black, citecolor=orange!50!black, urlcolor=blue!70!black]{hyperref}
\usepackage[capitalize, nameinlink]{cleveref}
\usepackage{graphicx} 
\usepackage{authblk}
\usepackage{standalone}

\usepackage{tikz}
\usepackage[shortlabels]{enumitem}

\DeclareMathOperator{\td}{\mathsf{td}}
\DeclareMathOperator{\pn}{\mathsf{pn}}

\DeclareMathOperator{\pw}{\mathsf{pw}}

\newtheorem{theorem}{Theorem}

\newtheorem{proposition}[theorem]{Proposition}

\newtheorem{corollary}[theorem]{Corollary}
\newtheorem*{remark}{Remark}

\usepackage{authblk}

\title{Excluding paths and bicliques\thanks{M.~C.~was supported by NSF Grant DMS-2348219, NSF Grant CCF-2505100, AFOSR grant  FA9550-25-1-0275 and a Guggenheim Fellowship.
J.~C.~was supported by NSF Grant DMS-2348219, NSF Grant CCF-2505100, AFOSR grant  FA9550-25-1-0275, and by the Fonds de recherche du Qu\'{e}bec Grant 321124.
M.~K.~was supported in part by the Slovenian Research and Innovation Agency (research program P1-0383 and research project N1-0370).
M.~M.~was supported in part by the Slovenian Research and Innovation Agency (I0-0035, research program P1-0285 and research projects J1-60012, J1-70035, J1-70046, and N1-0370).}}

\author[1]{Maria Chudnovsky}
\author[1]{Julien Codsi}
\affil[1]{Princeton University, Princeton, NJ 08544, USA\\ \normalfont\texttt{\{mchudnov,jc3530\}@math.princeton.edu}}

\author[2]{Matjaž Krnc} 
\author[2]{Martin Milanič}
\affil[2]{FAMNIT and IAM, University of Primorska, Koper, Slovenia\\ \normalfont\texttt{\{matjaz.krnc,martin.milanic\}@upr.si}}

\date{\normalsize\textit{Dedicated to the memory of Vera T. Sós.}}

\begin{document}

\maketitle
\begin{abstract}
Classes of graphs excluding a path and a biclique as induced subgraphs are extensively studied in the literature.
One of the key structural results for such graphs is a Ramsey-type result due to Galvin, Rival, and Sands~(1982), establishing the existence of a function  $f$ bounding the maximum length of a path in terms of clique number $\omega$.
We improve the best known bound on $f$ to a function that is a singly exponential in $\omega^c$, for some constant $c$, which we show is best possible, up to optimizing $c$.

Our approach also has consequences for treedepth.
In particular, we show that, for graphs excluding a path and a biclique as induced subgraphs, treedepth is bounded by a polynomial function of clique number.
In turn, this result implies that every hereditary graph class that admits a function bounding treedepth of graphs in the class in terms of clique number, admits a polynomial such function.
This gives a treedepth analogue of a recent result on pathwidth due to Hajebi (2025).
\end{abstract}

treedepth, hereditary graph class, clique-polynomiality.


\section{Introduction}

Galvin, Rival, and Sands~\cite{MR678167} in 1982 and independently Atminas, Lozin, and Razgon~\cite{MR2990848} in 2012 proved that the only reason for a graph to have a long path is the presence of a long induced path, a large clique, or a large induced biclique.
Denoting by $\omega(G)$ the clique number of $G$ and by $\pn(G)$ the \emph{path number} of $G$, that is, the maximum number of vertices in a (not necessarily induced) path in $G$, this result can be equivalently stated as follows.
 
\begin{theorem}[Galvin, Rival, and Sands~\cite{MR678167}; Atminas, Lozin, and Razgon~\cite{MR2990848}]\label{path-number}
For every two positive integers $s$ and $t$, there exists a function $f\colon \mathbb{N}\to \mathbb{N}$ such that every $\{P_s, K_{t,t}\}$-free graph $G$ satisfies $\pn(G)\le f(\omega(G))$.
\end{theorem}

Questions related to \Cref{path-number} were the subject of many recent papers, see Duron, Esperet, and Raymond~\cite{duron2024longinducedpathssparse} for an overview.
The best quantitative bound known up to now is a doubly exponential function of some power of the clique number.
This bound follows from the main result of Hunter, Milojević, Sudakov, and Tomon~\cite{Hunter2026} (see \Cref{sec:Hunter}).

We improve this result by showing that for graphs excluding a path and a biclique as induced subgraphs, the path number is bounded  above by a \textsl{singly exponential} function of some power of the clique number.

\begin{restatable}{theorem}{thmpathnumber}\label{thm:path number}
For every two positive integers $s$ and $t$ there exists a positive integer $c$ such that every $\{P_s,K_{t,t}\}$-free graph $G$ satisfies $\pn(G)\le 2^{\omega(G)^c}$.
\end{restatable}

Except for optimizing the dependency of $c$ on the excluded path and biclique, this result is best possible.
Indeed, we construct examples showing that the exponential dependency on the clique number in \Cref{thm:path number} cannot be avoided.

Our approach has consequences for treedepth, a graph parameter that plays a central role in the graph sparsity theory of Ne\v{s}et\v{r}il and Ossona de Mendez~\cite{SparsityBook}.
We show that, for graphs excluding a path and a biclique as induced subgraphs, treedepth is bounded by a polynomial function of the clique number.

\begin{restatable}{theorem}{thmtdbddomega}\label{thm:td-bdd-omega}
    For every two positive integers $s$ and $t$ there exists a positive integer $c$ such that every $\{P_s,K_{t,t}\}$-free graph $G$ satisfies 
    $\td(G)\le \omega(G)^c$.
\end{restatable}

\Cref{thm:td-bdd-omega} in turn implies the treedepth analogue of a result on pathwidth due to Hajebi~\cite{hajebi2025polynomialboundspathwidth}.
To state this result, we need a couple of definitions.
A \emph{graph parameter} is a function mapping graphs to integers that does not distinguish between isomorphic graphs.
A graph class is \emph{hereditary} if it is closed under vertex deletion.
A graph parameter $\rho$ is said to be \emph{clique-polynomial} if the following holds: for every hereditary graph class $\mathcal G$, if there exists a function $f$ such that $\rho(G)\le f(\omega(G))$ for every graph $G\in \mathcal G$, then there exists a polynomial function $p$ such that $\rho(G)\le p(\omega(G))$ for every graph $G\in \mathcal G$.
Hajebi~\cite{hajebi2025polynomialboundspathwidth} showed that pathwidth is clique-polynomial.
We show the same property for treedepth.

\begin{restatable}{theorem}{thmpolyboundsfortreedepth}\label{thm:poly-bounds-for-treedepth}
For every hereditary graph class $\mathcal G$, if there exists a function $f$ such that $\td(G)\le f(\omega(G))$ for every graph $G\in \mathcal G$, then there exists a polynomial function $p$ such that $\td(G)\le p(\omega(G))$ for every graph $G\in \mathcal G$.
\end{restatable}

Let us mention that analogous results do not hold for treewidth and chromatic number, as shown by Chudnovsky and Trotignon~\cite{MR4970654} and by Bria\'nski, Davies, and Walczak~\cite{MR4707561}, respectively, but do hold for degeneracy in some cases, due to a result of Bourneuf, Buci\'c, Cook, and Davies~\cite{MR4808986} and Girão and Hunter \cite{10.1093/imrn/rnaf025}.
Indeed, they showed that for every graph $H$, there is a polynomial~$p$ such that every $K_{t,t}$-subgraph-free graph with no induced subdivision of $H$ has average degree at most $p(t)$. 
By Ramsey's theorem~\cite{MR1576401}, this implies that for every graph $H$, there is a polynomial $p$ such that every $K_{t,t}$-free graph with no induced subdivision of $H$ has average degree at most $p(\omega(G))$ and hence, degeneracy  at most $p(\omega(G))$. 

A sufficient condition for a graph parameter to be clique-polynomial is that it is \textsl{awesome}, which is a property based on the so-called \textsl{independence variants} of parameters (see Bešter Štorgel et al.~\cite{storgel2025awesomegraphparameters}).
While several parameters, including vertex cover number, feedback vertex set number, and odd cycle transversal number, are awesome, treedepth and pathwidth are not~\cite{storgel2025awesomegraphparameters}. Nonetheless, \Cref{thm:poly-bounds-for-treedepth} and the aforementioned result of Hajebi show that these two parameters satisfy a necessary condition for awesomeness, namely clique-polynomiality.

\medskip
Our main results, \Cref{thm:path number,thm:td-bdd-omega}, contribute to the list of results on classes of graphs excluding a path and a biclique as induced subgraphs.
Such classes were extensively studied in the literature, from both structural and algorithmic points of view.
For example, Kami\'nski and Pstrucha~\cite{MR3958245} showed that  graphs excluding a path and a biclique as induced subgraphs admit finitely many subgraph-minimal obstructions to $H$-list-colorability (for any fixed graph~$H$), which implies the existence of polynomial-time certifying algorithms for $H$-list-colorability.
Dallard et al.~\cite{dallard2024treewidth} conjectured that such graphs have bounded \textsl{tree-independence number}, that is, they admit tree decompositions whose bags induce subgraphs of bounded independence number.
Following a number of partial results 
(see~\cite{chudnovsky2025, chudnovsky2026, Hilaire2026,blazej2026treeindependencenumberp5freegraphs}), the conjecture was proved recently in a more general context by Hajebi and Spirkl~\cite{hajebi2026}.
This implies that all good algorithmic properties enjoyed by graph classes with bounded tree-independence number (see~\cite{MR4664382,LMMORS2026,DBLP:conf/soda/Yolov18}) hold for graphs excluding a path and a biclique as induced subgraphs.
This generalizes a number of previous results (see, e.g.,~\cite{MR3575027,DBLP:conf/wg/HusicM19,MR2009287}).

\medskip
The paper is structured as follows.
After collecting the necessary definitions and tools in~\cref{sec:prelim}, we explain in \cref{sec:Hunter} how to obtain a quantitative version of \cref{path-number} from the main result of Hunter et al.~\cite{Hunter2026}.
In~\cref{sec:path-decompositions}, we derive two auxiliary results regarding bounds on treedepth and path number of $P_s$-free graphs in terms of pathwidth.
In \cref{sec:Galvin}, we prove \cref{thm:path number}, while in \cref{sec:treedepth}, we prove \cref{thm:td-bdd-omega,thm:poly-bounds-for-treedepth}.

\section{Preliminaries}\label{sec:prelim}

We denote by $\mathbb N$ the set of all nonnegative integers.
All logarithms in this paper are binary. 
Throughout this paper, all graphs are finite, simple, and undirected. 
For a graph $G$, we denote its \emph{order}, that is, the number of vertices, by $n(G)$. 
For a positive integer $n$, we denote by $P_n$ and $K_n$ the $n$-vertex path graph and the $n$-vertex complete graph, respectively. 
In this paper, a \emph{biclique} $K_{n,n}$ refers to a balanced complete bipartite graph of order $2n$.

For a subset of vertices $U \subseteq V(G)$ or a subgraph $U$ of $G$, we write $G - U$ to denote the graph obtained from $G$ by deleting all vertices in $U$ and their incident edges. 
Given a family $\mathcal{F}$ of graphs, we say that a graph $G$ is \emph{$\mathcal{F}$-free} if no induced subgraph of $G$ is isomorphic to a member of $\mathcal{F}$; furthermore, for a graph $F$, we say that $G$ is \emph{$F$-free} if it is $\{F\}$-free. 

For integers $k, \ell$, we denote by $[k,\ell]$ the set of integers greater than or equal to $k$ not exceeding $\ell$, and set $[k]=  [1,k]$.

A \emph{Hamiltonian path} in a graph $G$ is a path that visits every vertex of $G$. 
A graph $G$ is \emph{traceable} if it admits a Hamiltonian path.

A \emph{path decomposition} of a graph $G$ is a sequence $(X_1, \dots, X_r)$ of subsets of $V(G)$ called \emph{bags} satisfying the following two 
properties: 
\begin{enumerate}[(i)]
    \item for every edge $uv \in E(G)$, there exists an index $i \in [r]$ such that $\{u, v\} \subseteq X_i$, and \label{pw:prop-ii}
    \item for every $v\in V(G)$ there exist integers $k,\ell\in [r]$ with $k\le \ell$ such that for all $j\in [r]$, we have
    $v\in X_j$ if and only if $j\in [k,\ell]$.
    \label{pw:prop-iii}
\end{enumerate}
The \emph{width} of a path decomposition is defined as $\max_{1 \le i \le r} |X_i| - 1$.
The \emph{pathwidth} of $G$, denoted by $\pw(G)$, is the minimum width over all possible path decompositions of $G$.

A \emph{rooted tree} is a tree $T$ together with a distinguished vertex $r \in V(T)$ called the \emph{root}. 
A \emph{rooted forest} is a disjoint union of rooted trees. 
The \emph{transitive closure} of a rooted forest $F$ is the graph with vertex set $V(F)$ where two vertices are adjacent if and only if one is an ancestor of the other in $F$. 
A \emph{treedepth decomposition} of a graph $G$ is a rooted forest $F$ on the vertex set $V(G)$ such that 
$G$ is a subgraph of the transitive closure of $F$. 
The \emph{depth} of a rooted forest $F$ is the maximum number of vertices on any root-to-leaf path in $F$. The \emph{treedepth} of $G$, denoted by $\td(G)$, is the minimum depth over all treedepth decompositions of $G$.

\subsection{Tools}

For easier reference, we enumerate here several standard (folklore) facts that we will use, sometimes implicitly, throughout the paper (see, e.g., \cite{SparsityBook,MR1647486}).

\begin{enumerate}[label=\textbf{Fact \arabic*.}, ref=Fact \arabic*, leftmargin=*]
    \item  Let $G$ be a graph and $C_1, \dots, C_k$ its connected components. 
    Then $\pw(G) = \max_{1 \le i \le k} \pw(C_i)$.
    Similarly, $\td(G) = \max_{1 \le i \le k} \td(C_i)$ and $\pn(G) = \max_{1\le i\le k}\pn(C_i)$. \label{obs:components}
      \item Path number, pathwidth, treedepth, and clique number are all monotone under subgraphs.\label{fct:monotone} 
     \item Let $G$ be a graph containing an edge, and let $I$ be the set of isolated vertices in $G$. Then $\pn(G)=\pn(G-I)$, $\td(G)=\td(G-I)$, and $\pw(G)=\pw(G-I)$. \label{fct:ignoreIsolates}
    \item\label{fct:subgraph-interval} Let $G$ be a graph with a path decomposition $(X_1,\dots,X_r)$, and let $H$ be a connected subgraph of $G$. 
    Then there exist integers $k,\ell\in [r]$ with $\ell\le k$ such that for all $j\in [r]$, the bag $X_j$ intersects $H$ if and only if $j\in [\ell,k]$.
    \item \label{fct:clique-in-bag} For every graph $G$, every path decomposition $(X_1, \dots, X_r)$ of $G$,
    and every clique $K$ of $G$, there exists  $i \in [r]$ such that $K \subseteq X_i$.
  
\item \label{treedepth-path-number}
Let $G$ be a graph.
Then 
\[\omega(G)\le \pw(G)+1\le \td(G) \le \pn(G)\le 2^{\td(G)}-1\,.\]
\end{enumerate}

The following theorem is an immediate consequence of~\cite[Theorem 2.1]{TW18}.

\begin{theorem}[{Chudnovsky, Hajebi, and Spirkl~\cite{TW18}}] \label{CHS}
For all  positive integers $r$, $s$,  and $t$, there exists an integer $k$ such that every $\{P_r, K_s, K_{t,t}\}$-free graph has pathwidth at most $k$.
\end{theorem}
\Cref{CHS} also follows immediately from \cite{MR2990848,MR678167} and \cite{GM1}.
\begin{theorem}[Hajebi~\cite{hajebi2025polynomialboundspathwidth}]\label{thm:hajebi1}
For every hereditary graph class $\mathcal G$, if there exists a function $f$ such that $\pw(G)\le f(\omega(G))$ for every graph $G\in \mathcal G$, then there exists a polynomial function $p$ such that $\pw(G)\le p(\omega(G))$ for every graph $G\in \mathcal G$.
\end{theorem}

Together \cref{CHS,thm:hajebi1} imply:

\begin{corollary} \label{polypathwdith}
For every two positive integers $s$ and $t$ there exists a positive integer $c$ such that the following holds.
Every $\{P_s,K_{t,t}\}$-free graph $G$ satisfies $\pw(G) \leq {\omega(G)}^c$.
\end{corollary}

\begin{proof}
Fix positive integers $s$ and $t$, and let $\mathcal G$ be $\{P_s,K_{t,t}\}$-free graphs.
\cref{CHS} implies that the pathwidth of every graph in $\mathcal G$ is bounded by some function of its clique number.
By \cref{thm:hajebi1}, the pathwidth of every graph in $\mathcal G$ is bounded by a polynomial function $p$ of its clique number. 
Since a graph with $\omega(G) = 1$ satisfies $\pw(G) = 0$, it remains to justify that there exists some positive integer $c$ such that $p(x)\le x^c$ for all $x\ge 2$.
Indeed, after dropping the terms with negative coefficients, we observe that $p$ is upper-bounded by a polynomial of some degree $d$ in which all nonzero terms have the same positive coefficient $a$.
Any such polynomial is upper bounded by a monomial of the form $ax^{d+1}$, which in turn is bounded from above by $x^c$ for all $x\ge 2$, where $c = d+1+\lceil \log a\rceil$.
\end{proof}

\subsection{Transitive closures of binary trees}

For a positive integer $k$, denote by $T^+_k$ the transitive closure of a complete binary tree of depth $k$. 
See \cref{fig:dense-tree} for an example.
The family $\{T^+_k\}_{k>0}$ will serve as a tightness example for several of our claims. 

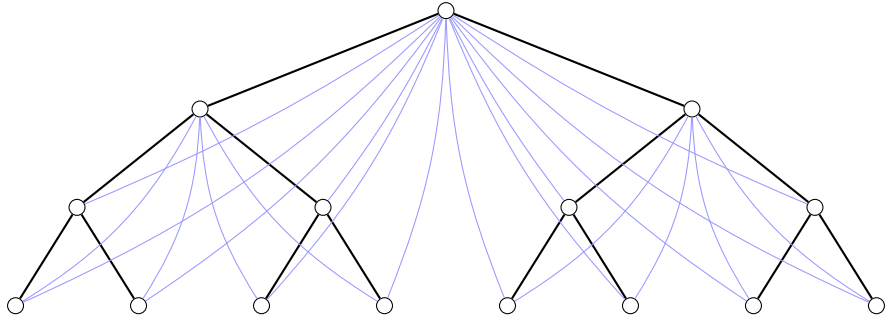
\begin{figure}[htbp]
    \centering
\begin{tikzpicture}[
    node style/.style={circle, draw=black, fill=white, inner sep=0pt, minimum size=6pt},
    tree edge/.style={black, thick},
    clique edge/.style={blue!40!white, thin} 
]

\def\totalwidth{13}
\def\levelsep{1.3}

\foreach \l in {0,1,2,3} {
    \pgfmathtruncatemacro{\numNodes}{pow(2,\l)}
    \pgfmathtruncatemacro{\lastNode}{\numNodes-1}
    \foreach \i in {0,...,\lastNode} {
        \pgfmathsetmacro{\xpos}{\totalwidth * (\i + 0.5) / \numNodes}
        \pgfmathsetmacro{\ypos}{-\l * \levelsep}
        
        \node[node style] (N\l_\i) at (\xpos, \ypos) {};
    }
}

\foreach \l in {0,1,2} {
    \pgfmathtruncatemacro{\numNodes}{pow(2,\l)}
    \pgfmathtruncatemacro{\lastNode}{\numNodes-1}
    \foreach \i in {0,...,\lastNode} {
        \pgfmathtruncatemacro{\childA}{2*\i}
        \pgfmathtruncatemacro{\childB}{2*\i + 1}
        \pgfmathtruncatemacro{\nextLevel}{\l+1}
        
        \draw[tree edge] (N\l_\i) -- (N\nextLevel_\childA);
        \draw[tree edge] (N\l_\i) -- (N\nextLevel_\childB);
    }
}

\foreach \i in {0, 1} { \draw[clique edge] (N0_0) to[bend right=-5] (N2_\i); }
\foreach \i in {2, 3} { \draw[clique edge] (N0_0) to[bend left=-5]  (N2_\i); }

\foreach \i in {0, 1, 2, 3} { \draw[clique edge] (N0_0) to[bend right=-10] (N3_\i); }
\foreach \i in {4, 5, 6, 7} { \draw[clique edge] (N0_0) to[bend left=-10]  (N3_\i); }

\foreach \i in {0, 1} { \draw[clique edge] (N1_0) to[bend right=-15] (N3_\i); }
\foreach \i in {2, 3} { \draw[clique edge] (N1_0) to[bend left=-15]  (N3_\i); }

\foreach \i in {4, 5} { \draw[clique edge] (N1_1) to[bend right=-15] (N3_\i); }
\foreach \i in {6, 7} { \draw[clique edge] (N1_1) to[bend left=-15]  (N3_\i); }

\end{tikzpicture}
    \caption{The graph $T^+_k$ is the transitive closure of a complete binary tree of depth $k$. 
    The figure depicts~$T^+_4$.\label{fig:dense-tree}}
\end{figure}
\newpage

\begin{proposition}\label{prop:dense-trees}
Let $k$ be a positive integer. Then:
\begin{enumerate}[label=$\mathrm{(\alph*)}$]
    \item For $k>1$, the graph $T^+_k$ is isomorphic to the graph obtained from the disjoint union of two copies of $T^+_{k-1}$ by adding a universal vertex. \label{P-inductive-def}
    \item $T^+_k$ is traceable, that is, $\pn(T^+_k) = |V(T^+_k)| = 2^{k}-1$.
    \item $T^+_k$ is $\{P_4,K_{2,2}\}$-free.
    \item $\omega(T_k^+)=\pw(T_k^+) + 1 = \td(T^+_k) = k$.
\end{enumerate}
\end{proposition}

\begin{proof}
For completeness we offer short proofs in turn.
\begin{enumerate}[\text(a)]
\item Immediate from the definition and the fact that for $k>1$, the complete binary tree of depth $k$ is obtained from the disjoint union of two complete binary trees of depth $k-1$ by adding a new vertex as a root and making it adjacent to the roots of the two smaller trees.
    \item By induction on $k$. 
For $k>1$, due to \ref{P-inductive-def}, the root $r$ of $T^+_k$ is adjacent to all vertices in its left and right subtrees, both of which are isomorphic to $T^+_{k-1}$. 
Concatenating a Hamiltonian path of the left subtree, the root $r$, and a Hamiltonian path of the right subtree yields a Hamiltonian path of $T^+_k$.
    
\item By induction on $k$, using \ref{P-inductive-def} and the fact that the operations of disjoint union and addition of a universal vertex preserve $\{P_4,K_{2,2}\}$-freeness.
    
\item By \ref{treedepth-path-number}, it suffices to show that $\td(T^+_k) \le k \le \omega(T^+_k)$.
The former inequality follows from the fact that the underlying binary tree of depth $k$ is a valid treedepth decomposition.
For the latter one, observe that any root-to-leaf path in the underlying binary tree of depth $k$ forms a clique of size $k$ in~$T^+_k$.\qedhere
\end{enumerate}    
\end{proof}

\section{Comparison with Hunter et al.}\label{sec:Hunter}

The main result of Hunter et al.~\cite{Hunter2026} hides the dependency on the size of the excluded biclique in the $\Omega(\cdot)$ notation.
The variant below can be obtained from the proof. 

\begin{theorem}\label{thm:Hunter}
Let $G$ be a $K_{r,r}$-subgraph-free graph containing a path on $n$ vertices, where $n\geq r\geq 2$.
Then $G$ contains an induced path of length at least $\Omega \big(\frac{\log \log n}{r\log\log\log n}\big)$.
\end{theorem}

We now explain why \Cref{thm:Hunter} leads to a quantitative variant of \Cref{path-number} establishing a doubly exponential bound on the path number in terms of some power of the clique number.

\begin{theorem}
For every two positive integers $s$ and $t$ there exists a positive integer $c$ such that every $\{P_s,K_{t,t}\}$-free graph $G$ satisfies $\pn(G)<2^{2^{(\omega(G)+1)^c}}$.
\end{theorem}

\begin{proof}
\Cref{thm:Hunter} implies that there exists a positive integer $a$ such that for every $n\ge r\ge 2$, every graph $G$ that does not contain $K_{r,r}$ as a subgraph, and which contains a path on $n$ vertices, contains an induced path of length at least $a(\log\log n)^{0.99}/r$.

Fix $s$ and $t$ and let $G$ be a $\{P_s,K_{t,t}\}$-free graph.
Let $r = (\omega(G)+1)^t$.
Since $G$ is $K_{t,t}$-free, it does not contain $K_{r,r}$ as a subgraph.
Indeed, if $G$ contains a $K_{r,r}$, with parts $A$ and $B$, then at least one of the two parts, say $A$, has independence number less than $t$, but then $G[A]$ would have to contain a clique of size $\omega(G)+1$, a contradiction.
Here we are using  the bound of Erd\"os and Szekeres~\cite{MR1556929}
that every
graph on at least $\binom{s+t-2}{t-1}$ vertices contains either a clique of size $s$ or a stable set of size $t$, and the observation that $\binom{s+t-2}{t-1} \leq \binom{s+t-1}{t} \leq s^t$.

We claim that $\pn(G)< 2^{2^{(\omega(G)+1)^c}}$ where $c$ is any integer satisfying 
\[{c \ge (t+ \log s - \log a)/0.99}\,.\]
Suppose for a contradiction that $G$ contains a path on $n$ vertices, where \[n\ge 2^{2^{(\omega(G)+1)^c}}\,.\]
By assumption, $G$ contains an induced path on at least $a(\log\log n)^{0.99}/r$ vertices.
The inequality $n\ge 2^{2^{(\omega(G)+1)^c}}$ yields ${\log\log n \ge (\omega(G)+1)^c}$, and consequently  
\begin{align*}
    a\frac{(\log\log n)^{0.99}}{r} &\ge a\frac{(\omega(G)+1)^{0.99c}}{r}\\
    &= a(\omega(G)+1)^{0.99c-t}\\
    & \ge a2^{0.99c-t} \ge s\,.
\end{align*}
Hence, $G$ contains an induced path on $s$ vertices, a contradiction.
\end{proof}

\section{Harnessing path decompositions}\label{sec:path-decompositions}

In this section, we bound treedepth and path number of $P_s$-free graphs in terms of pathwidth.
We start by proving a linear relationship between treedepth and pathwidth.

\begin{theorem}
    \label{thm:pwtotd}
Let $s\ge 3$ be an integer and let $G$ be $P_s$-free graph with at least one edge. 
Then $\td (G) \leq (s-1)\cdot \pw(G)$.
\end{theorem}
      \begin{proof} 
First observe that the inequality holds for graphs with $\pw(G)=1$. 
Indeed, if $\pw(G)=1$, then $G$ is a forest, so $\pn(G) < s$ and \ref{treedepth-path-number} implies
      \[\td(G) \le \pn(G)\le (s-1)\pw(G)\,.\]
The proof for the general case is by induction on $n(G)$.
      If $n(G)= 2$, we are done since $\pw(G)=1$.

      For the rest of the proof let $n(G) \ge 3$ and $\pw(G)\ge 2$. Suppose first that $G$ is disconnected. 
      Denote by $I$ the set of isolated vertices in $G$, and let $G_1,\ldots, G_k$ for some $k\ge 1$ be the non-trivial components of $G$. 
      Then, the induction hypothesis yields $\td (G_i) \leq (s-1) \pw(G_i)$ for each $i\in [k]$.
      Hence, by \ref{obs:components} and \ref{fct:ignoreIsolates}, we obtain 
      \begin{align*}
      \td(G) &= \td(G-I)
      = \max_{1\le i\le k}\td(G_i)\\
      &\le (s-1) \left(\max_{1\le i\le k}\pw(G_i)\right) \\[2mm]
      &= (s-1)\pw(G-I)
      = (s-1)\pw(G)\,,    
      \end{align*}
      as desired. 

      We may therefore assume that $G$ is connected. 
      Let $(X_1, \dots, X_m)$ be a path decomposition of $G$ of width $\pw(G)$. 
      Let $a_1 \in X_1$ and $a_m \in X_m$, and let $P$ be a shortest path in $G$ from $a_1$ to $a_m$.
      Since $P$ is an induced path and $G$ is $P_s$-free, $|V(P)| \le s-1$.
      Furthermore, \ref{fct:subgraph-interval} implies that $V(P) \cap X_i \neq \emptyset$ for every $i \in [m]$.
      It follows that $\pw(G - P) \le \pw(G)-1$.
      Let $F$ be a treedepth decomposition of $G- P$ with depth equal to $\td(G- P)$ and let $R$ be the set of roots of $F$.
      Let $T$ be the tree obtained from the disjoint union of $F$ and $P$ by adding an edge between $a_m$ and every vertex in $R$, and let $a_1$ be the root of $T$. 
      Then, $T$ is a treedepth decomposition of $G$ with depth at most $\td(G - P)+(s-1)$, implying that $\td(G) \leq \td(G - P)+(s-1)$. 
      
    If $G-P$ has at least one edge, then the induction hypothesis implies that $\td(G- P) \leq (s-1)(\pw(G)-1)$, and thus
    \[\td(G) \leq \td(G - P)+(s-1) \leq (s-1)\pw(G)\,,\] 
    as desired.
    We may thus assume that $G-P$ is edgeless.
    Then, ${\td(G - P)\le 1}$ and, since $\pw(G)\ge 2$ and $s\ge 3$, we obtain that 
      \[\td(G) \leq \td(G - P)+(s-1) \leq s \le (s-1)\pw(G)\,,\]
      as desired.
\end{proof}

\cref{thm:pwtotd} and the inequality $\pn(G)\le 2^{\td(G)}-1$ (see \ref{treedepth-path-number}) imply that for every integer $s\ge 3$, every $P_s$-free graph with at least one edge satisfies \[\pn(G) \le 2^{(s-1)\pw(G)}-1\,.\] 
We recently learned of a better bound, obtained by a similar approach as the one used in the proof of \cref{thm:pwtotd}, due to Hilaire and Raymond \hbox{in~\cite[Theorem 6]{MR4540912}}: 
for every integer $s\ge 3$, every $P_s$-free graph with at least one edge satisfies \[\pn(G) \le (3s-3)^{\pw(G)+1}\,.\] 
Using a similar strategy, we obtain a further improvement, as follows.
 
\begin{theorem}
    \label{thm:pwtopn}
Let $s\ge 3$ be an integer and let $G$ be $P_s$-free graph with at least one edge. 
Then \[\pn (G) \leq s^{\pw(G)}-1\,.\]
\end{theorem}

 \begin{proof}
We again start by verifying that, in the case $\pw(G)=1$, we have a forest and the fact ${\pn(G)<s}$ implies our claim.
    We proceed by induction on $n(G)$ where the base case (i.e., $n(G)=2$) is settled due to $\pw(K_2)=1$.

Suppose now that $n(G) \ge 3$.
    Consider first the case  where $G$ is disconnected. 
    Denote by $I$ the set of isolated vertices in $G$ and again recall \ref{fct:ignoreIsolates}, in particular, $\pw(G)=\pw(G-I)$ and $\pn(G)=\pn(G-I)$. 
Denote by $G_1,\ldots, G_k$ the non-trivial components of $G$, and observe by the induction hypothesis that $\pn(G_i) \leq s^{\pw(G_i)}-1$ for each $i\in [k]$.
      By \ref{obs:components}, we thus obtain 
      \begin{align*}
      \pn(G) &= \pn(G-I)
      = \max_{1\le i\le k}\pn(G_i)\\
      &\le  s^{\max_{1\le i\le k}\pw(G_i)}-1\\[2mm]
      &= s^{\pw(G-I)}-1
      = s^{\pw(G)}-1\,,    
      \end{align*}
as desired.

Suppose now that $G$ is connected, let $(X_1, \dots, X_m)$ be a path decomposition of $G$ realizing $\pw(G)$, and let $P$ be a shortest path connecting some vertex of $X_1$ with some vertex of $X_m$.
We again have and $P \cap X_i \neq \emptyset$ for every $i \in \{1, \dots, m\}$, and also $|V(P)| < s$, since $G$ is $P_s$-free. 
In particular, $\pw(G - P) \le \pw(G) - 1$.
Note that we may assume that $G-P$ has at least one vertex, since otherwise $G$ is a path and has pathwidth~$1$.

Let $C_1,\dots,C_{\ell}$ be the connected components of $G-P$.
Let $Q$ be a path in $G$.  Consider how $Q$  alternates between vertices of $P$ and the connected components of $C_1, \dots, C_l$. 
Since $P$ contributes at most $s-1$ vertices to $Q$, the path $Q$ can traverse at most $s$ such components of $G-P$. 
Therefore, we can bound $\pn(G)$ by
\[\pn(G) \le s \cdot \max_{i\in [\ell]} \pn(C_i) + (s - 1) 
    \le s \cdot  \pn(G - P) + s  - 1\,.\]
If $G - P$ is edgeless, then $\pn(G - P)=1$, which implies  $\pn(G)\le  2s- 1\le s^{\pw(G)}-1$ due to $\pw(G)\ge 2$ and $s\ge 3$.
Otherwise, the induction hypothesis on $G-P$ and the fact that $\pw(G - P) \le \pw(G) - 1$ imply that
\begin{align*}    
    \pn(G)&\le s \left( \left( s^{\pw(G)-1} - 1 \right) + 1 \right) - 1 
    = s^{\pw(G)} - 1,
\end{align*}
as desired.
\end{proof}

\noindent 
We observe that exponential dependence on the pathwidth cannot be avoided due to, for instance, transitive closures of binary trees which are $P_4$-free (also see \cref{fig:dense-tree}). 
Indeed, for any positive integer $k$, we have $\pw(T^+_k)=k$ while  $\pn(T^+_k)=2^{k+1}-1$ (see \cref{prop:dense-trees}).

\begin{remark}\label{rem:MR4540912}
    The aforementioned result  in~\cite[Theorem 6]{MR4540912} is stated as: 
    \begin{quote}
        For every positive integer $k$, if $G$ is a graph with pathwidth less than $k$ that has a path of order $n$, then $G$ has an induced path of order at least $\frac{1}{3}n^{1/k}$.
    \end{quote}
    \Cref{thm:pwtopn} leads to the following improved bound: 
    \begin{quote}
        For every integer $k\ge 2$, if $G$ is a graph with pathwidth less than $k$ that has a path of order $n$, then $G$ has an induced path of order at least $(n+1)^{1/(k-1)}-1$.
    \end{quote}
\end{remark}

\section{Improving a Ramsey-type result}\label{sec:Galvin}

Using the exponential relationship between the path number of a $P_s$-free graph and its pathwidth given by \cref{thm:pwtopn}, we now derive one of our main results, the following quantitative improvement of \cref{path-number}.

\thmpathnumber*

\begin{proof}
Fix two positive integers $s$ and $t$.
We may assume that $s\ge 2$.
By \Cref{polypathwdith}, there exists a positive integer $d$ such that every  $\{P_s,K_{t,t}\}$-free graph $G$ satisfies $\pw(G) \leq {\omega(G)}^d$.
Let $G$ be a $\{P_s,K_{t,t}\}$-free graph.
We may assume that $\omega(G)\ge 2$, since otherwise $\pn(G)\le 1\le 2^{{\omega(G)}^c}$ for every  integer $c$.
Note that $s\ge 3$, since $G$ is $P_s$-free but has at least one edge.
Hence, by \Cref{thm:pwtopn}, $\pn (G) \leq s^{\pw(G)}$.
Let $c=  d+\lceil\log\log s\rceil$.
Then $\log s\le 2^{c-d}\le {\omega(G)}^{c-d}$ and hence,
$(\log s)\cdot {\omega(G)}^d\le {\omega(G)}^c$, which together with the inequality ${\pw(G) \leq {\omega(G)}^d}$ implies
\[\pn (G) \leq s^{\pw(G)} \leq s^{{\omega(G)}^d} \leq 2^{(\log s)\cdot {\omega(G)}^d}\le 2^{{\omega(G)}^c}\,,\] 
as desired.
\end{proof}

Exponential dependence on the clique number cannot be avoided, e.g., due to the family of transitive closures of binary trees $\{T^+_k\}_{k>0}$ (also see \cref{fig:dense-tree}), where $\omega(T^+_k)=k$ while $\pn(T^+_k)=2^{k+1}-1$ by \cref{prop:dense-trees}.
Furthermore, the constant $c$ must depend on $s$ and $t$, as can be seen, for instance, by looking at the family of path graphs.

\section{Treedepth is clique-polynomial}\label{sec:treedepth}

In what follows, we make use of the other main result from \cref{sec:path-decompositions}, the linear relationship between treedepth and pathwidth for $P_s$-free graphs given by \cref{thm:pwtotd}, to prove \cref{thm:td-bdd-omega}.

\thmtdbddomega*
  
\begin{proof}
Fix $s$ and $t$.
By \Cref{polypathwdith}, there exists a positive integer $d$ such that ${\pw(G) \leq {\omega(G)}^d}$, where $G$ is an arbitrary $\{P_s,K_{t,t}\}$-free graph.
We may assume that $G$ has at least one edge, since otherwise $\td(G)\le \omega(G)^d$.
Note that $s\ge 3$, since $G$ is $P_s$-free but has at least one edge.
Hence, \Cref{thm:pwtotd} implies that ${\td (G) \leq (s-1)\cdot \pw(G)}$.
Since $\pw(G) \leq {\omega(G)}^d$, we obtain that ${\td (G) \leq 
(s-1)\cdot {\omega(G)}^d}$.
Similarly as in the proof of \cref{polypathwdith}, we conclude that there exists a positive integer $c$ depending only on $s$ and $d$ (and hence, only on $s$ and $t$) such that $\td(G)\le \omega(G)^c$, as desired.
\end{proof}

Finally, we derive from \cref{thm:td-bdd-omega} the analogue of \cref{thm:hajebi1} for treedepth, that is, we show that treedepth is clique-polynomial.

\thmpolyboundsfortreedepth*
\begin{proof}
By \cref{thm:td-bdd-omega} it is enough to prove that every graph in $\mathcal G$ is $\{P_s,K_{t,t}\}$-free.
To this end let $\mathcal G$ be a $(\td,\omega)$-bounded graph class.
Fix a function $f\colon \mathbb{N}\to \mathbb{N}$ such that $\td(G)\le f(\omega(G))$ for all graphs $G\in \mathcal{G}$.
We may assume that $f$ is non-decreasing, since otherwise we could redefine $f(k)$ for $k\in \mathbb{N}$ by setting it to $\max_{i\in [k]}f(i)$.
Let $s = t = 2^{f(2)}$.
We claim that $\mathcal{G}$ excludes $P_s$.
Suppose for a contradiction that $P_s\in \mathcal{G}$.
Since $\omega(P_s) \le 2$ and $f$ is non-decreasing, we have that $\td(P_s)\le f(2)$. 
By \ref{treedepth-path-number}, it holds that $s = \pn(P_s)\le 2^{f(2)}-1$, a contradiction.
Furthermore, we claim that $\mathcal{G}$~excludes $K_{t,t}$.
Suppose for a contradiction that $K_{t,t}\in \mathcal{G}$.
Since $\omega(K_{t,t}) = 2$, we have that $\td(K_{t,t})\le f(2)$. 
Then, \ref{treedepth-path-number} implies that $2^{f(2)+1} = 2t = \pn(K_{t,t})\le 2^{f(2)}-1$, a contradiction.
Since $\mathcal{G}$ is hereditary, every graph in $\mathcal{G}$ is $\{P_s,K_{t,t}\}$-free, as desired. 
\end{proof}

\affil[1]{Princeton University}
\affil[2]{FAMNIT and IAM, University of Primorska, Koper, Slovenia}

\end{document}